\documentclass[A4paper,12pt]{article}

\usepackage{verbatim,amsfonts,amssymb}
\usepackage[utf8]{inputenc}
\usepackage[english]{babel}
\usepackage{graphicx}
\usepackage{xspace}
\usepackage{amsthm}
\usepackage{amsmath}

\usepackage{amsmath,graphicx,amsfonts}

\usepackage{tikz}
\usetikzlibrary{positioning}

\def\Z{\mathbb{Z}}
\def\R{\mathbb{R}}

\def\<{\langle}
\def\>{\rangle}

\DeclareMathOperator{\Vol}{Vol}
\DeclareMathOperator{\V}{V}

\usepackage{fancyhdr}
\theoremstyle{definition}
\newtheorem{defn}{Definition}[section]

\newtheorem{exmp}{Example}[section]
\newtheorem*{convention*}{Convention}
\newtheorem*{remark*}{Remark}

\newtheorem{theorem}{Theorem}
\newtheorem{lemma}{Lemma}

\usepackage{hyperref}

\begin{document}

\title{Polyhedra of small relative mixed volume}
\author{Ziyi Zhang\thanks{zizh@umich.edu\newline The research is done in the framework of the REU programme at the HSE Univeristy, Moscow:\newline \url{https://math.hse.ru/en/reu}}}

\maketitle

\begin{abstract}
We classify all tuples of lattice polyhedra of relative mixed volume 1 and all minimal (by inclusion) tuples of polyhedra of relative mixed volume 2.  We also prove a conjecture by A. Esterov, which states that all tuples with finite relative mixed volume are contained in one of finitely many ones that are minimal by inclusion.
\end{abstract}

\section{Introduction}

Mixed volumes of bounded lattice polytopes are extensively studied due to their relation to algebraic geometry. This relation is based on the Kouchnirenko--Bernstein formula for the number of roots of a general system of polynomial equations \cite{bernst}.

In particular, it is known that tuples of lattice polytopes of a given mixed volume in the space of a given dimension admit an essentially finite classification up to automorphisms of the lattice, see \cite{galois}. This classification is explicitly described for the lattice mixed volume one in any dimension \cite{mv1} and for the lattice mixed volume up to four in dimensions up to three \cite{mv4} (the lattice mixed volume up to four is especially interesting due to its relation to systems of equations solvable by radicals, see \cite{rad} and \cite{galois}).

Similarly, there is a version of the Kouchnirenko--Bernstein formula for the multiplicity of an isolated root of a system of analytic equations in terms of the so called relative mixed volume of pairs of lattice polytopes, see \cite{rel5}, \cite{rel6} and Section \ref{Smv} below. The relative mixed volume intricately alters the properties of the classical one. For instance, it satisfies the inverse Aleksandrov--Fenchel inequality \cite{af}.

This paper initiates the classification of tuples of polyhedra with a given relative mixed volume. This problem differs from its counterpart for bounded polytopes in that there are infinitely many tuples of a given relative mixed volume in a given dimension. However, we prove a conjecture by A. Esterov, instead that each of these tuples is contained in one of finitely many ones that are minimal by inclusion.

In every dimension, we obtain the classification of all (infinitely many) tuples of lattice volume 1 and (finitely many) minimal tuples of lattice volume 2. See Sections \ref{Smv1} and \ref{Smv2} respectively. These results outrun the current knowledge of the bounded case (where the classification of tuples of lattice mixed volume 2 is unknown starting from dimension 4). Finally, in section \ref{final}, we generalize the previous results and prove the conjecture that there are finitely many minimal tuples for any finite volume and dimension. 

Our proof has to be completely different from the proof for bounded polytopes, because the latter relies on the Aleksandov--Fenchel inequality, which does not hold in our setting (see above). One important tool that we use is the notion of so called interlaced polyhedra, see \cite{interl}; a version for bounded polytopes is recently given in \cite{interlb}.

\section{Relative mixed volume}\label{Smv}

Before we start introducing mixed volume of polyhedra, let us recall a few definitions. 
\begin{defn}
A \textit{polyhedron} in $\mathbb{R}^n$ is an intersection of finitely many closed half-spaces. A \textit{lattice polyhedron} is a polyhedron with all vertices contained in $\mathbb{Z}^n$. 
\end{defn}

In this research, we primarily focus on tuples of \textit{pairs of polyhedra} that differ by a bounded set and the difference between their volumes. \\

First, we introduce the $k$-dimensional lattice volume of polytopes.\\

Given a $k$-dimensional rational subspace $V$ in $\R^n$, we will define the $k$-dimensional lattice volume of polytopes in $V$ as follows. Denote $L = \Z^n \cap V$ and choose a linear map $f: \R^k \to \R^n$, such that $f(\mathbb{Z}^k) = L$.
\begin{defn}\label{lat_vol}
The \textit{k-dimensional lattice volume} of a polytope $K$ in the subspace $V$ is $k!$ times the metric volume of its preimage $f^{-1}(K) \subseteq \R^k$.
\end{defn}
Notice that the lattice volume as defined above is independent of the map chosen, because maps preserving the lattice preserve the volume as well. The lattice volume of a lattice polytope is always an integer, and the lattice volume of the standard simplex in $\mathbb{R}^n$ is equal to $1$.\\

\begin{minipage}{0.65\textwidth}
\begin{exmp}
Consider the intersection of the line $x+y=0$ and the square $[-1,1]^2$ laying in a plane, we will define the linear map $f: \mathbb{R} \rightarrow \mathbb{R}^2$, such that $f(1) = (1,-1)$. Correspondingly, we see the preimage of the intersection is $[-1,1]$, which has metric volume $2$, so it also has lattice volume $ 2$. As a result, we conclude that the intersection of the line $x+y=0$ and the square $[-1,1]^2$ has lattice volume $2$.\\
Now, consider  the intersection of the plane $x+y+z=0$ and the cube $[-1,1]^3$. We define the linear map $f: \mathbb{R}^2 \rightarrow \mathbb{R}^3$ as $f(1,0) = (1,0,-1)$ and $f(0,1) = (1,-1,0)$. Take the preimage of the intersection, we get a hexagon with the metric volume $3$, i.e. the lattice volume $6$. Correspondingly, the intersection of the plane $x+y+z=0$ and the cube $[-1,1]^3$ has lattice volume $6$. 
\end{exmp}
\end{minipage}
\begin{minipage}{0.3\textwidth}
\includegraphics{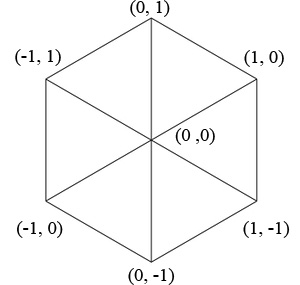}
\end{minipage}\\

\begin{convention*}
Note that all references to volume in the rest of the text are referring to lattice volume, unless otherwise stated.
\end{convention*}

\begin{defn}
For any two sets $K,L \in \mathbb{R}^n$, we call $K+L := \{x+y\;|\; x \in K, y \in L\}$ the \textit{Minkowski sum}, or briefly, the sum of $K$ and $L$.
\end{defn}

\begin{defn}\label{mixed_top}
Let $\mathcal{M}$ denote the set of all convex polytopes in $\mathbb{R}^n$, considered as a semigroup with respect to the Minkowski addition. The \textit{mixed volume} is the unique symmetric multilinear function $\Vol: \underbrace{\mathcal{M} \times ... \times \mathcal{M}}_{n} \rightarrow \mathbb{R}$ such that $\Vol(K,...,K)$ is equal to the volume of $K$ for all $K \in \mathcal{M}$.
\end{defn}

\begin{remark*}
The mixed volume may be defined in different ways (see e.g. \cite{thm2}, Definition IV.3.3), but it is known to satisfy the three aforementioned properties (see e.g. \cite{thm2}, Lemmas IV.3.4, 3.5, 3.6), and these three properties together uniquely define it, because they imply the following formula for the mixed volume:
\begin{align*}
\Vol(K_1,\ldots, K_n) = \sum_{I \subseteq \{1,\ldots, n\}}(-1)^{n-|I|}\Vol(\sum_{i \in I} K_i)
\end{align*}
where $K_1, \ldots, K_n$ are convex polytopes.
(see e.g. [GE91], Theorem IV.3.7).
\end{remark*}

\begin{remark*}
Notice that if $K_1, \ldots, K_n$ are \textit{lattice polytopes} in $\mathbb{R}^n$, then  $\Vol(K_1,\ldots, K_n)$ is an integer, because so is the right hand side of the above formula.
\end{remark*}

We define the sum of pairs of polyhedra componentwise. For pairs of polyhedra $(A_1, B_1)$, $(A_2, B_2)$, the sum of the pairs is defined by $(A_1,B_1)+(A_2,B_2)=(A_1+A_2,B_1+B_2)$. 

\begin{exmp}
Figure 1 is an example of the Minkowski sum of two pairs of polyhedra. As one can see, $A = C + \mathrm{conv}(\{(2,0),(0,2)\})$, $B = C + \mathrm{conv}(\{(3,0),(1,1),(0,3)\})$, where $C$ represents the first quadrant. Then, by definition of Minkowski sum, $A+B = C+C+\mathrm{conv}(\{(2,0)+(3,0),(0,2)+(3,0),(2,0)+(1,1),(0,2)+(1,1),(2,0)+(0,3),(0,2)+(0,3)\}) = C+\mathrm{conv}(\{(5,0),(3,1),(1,3),(0,5)\})$
\end{exmp}

\begin{figure}[h]
    \centering
    \includegraphics[width=\textwidth]{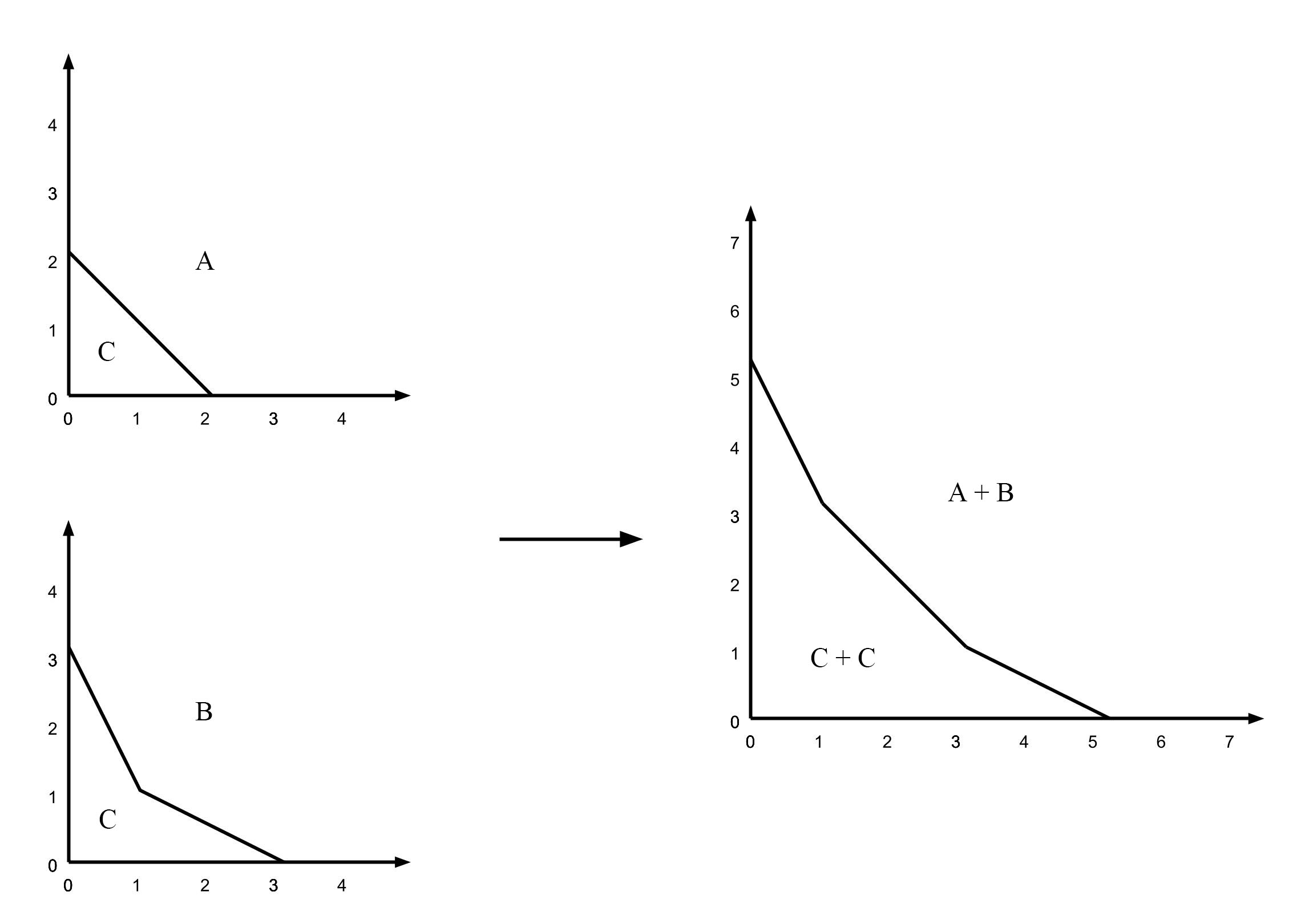}
    \caption{Two pairs of polyhedra underwent Minkowski addition, where $C$ represents first quadrant. In this case, $A = \mathrm{conv}(\{(2, 0),(0, 2)\})+C$ and $B = \mathrm{conv}(\{(3, 0),(1, 1),(0, 3)\})+C$. As such, by taking the Minkowski sum of pairs of polyhedra, we have $(C,A)+(C,B) = (C+C,A+B)$. In this case, we have $C+C=C$, so $(C,A)+(C,B) = (C,A+B)$}
    \label{fig:my_label}
\end{figure}

\begin{defn}\label{support}
Let $K \subset \mathbb{R}^n$ be a convex polyhedron  (not necessarily bounded). Its \textit{support function} $K(\cdot)$ is defined as 
$$K(\gamma) = \inf_{x \in K} \gamma(x)$$
for every covector $\gamma \in (\mathbb{R}^n)^*$. The set $\{\gamma | K(\gamma) > -\infty \} \subset (\mathbb{R}^n)^*$ is called the support cone of $K$.
\end{defn}

\begin{defn}
For a convex rational polyhedral cone $\Gamma$, denote  $\mathcal{M}_{\Gamma}$ as the set of pairs of convex polyhedra $(A,B)$ such that $A$ and $B$ have the support cone $\Gamma$ and the symmetric difference $A\Delta B:=(A \backslash B)\cup (B \backslash A)$ is bounded.
\end{defn}

Note that the set $M_\Gamma$ is a semigroup with respect to the componentwise summation $(A_1,B_1)+(A_2,B_2)=(A_1+A_2,B_1+B_2)$. 

\begin{defn}\label{mixed}[see \cite{rel5} (the first definition) and \cite{rel6} (Definition 7 in the journal version, Definition 1.1 in the extended version)]
The \textit{volume} $\V(A,B)$ of a pair of polyhedra $(A,B) \in \mathcal{M}_{\Gamma}$ is the difference of the volumes of the sets $A \backslash B$ and $B \backslash A$. The \textit{mixed volume} of pairs of polyhedra with support cone $\Gamma \subset (\mathbb{R}^n)^*$ is the symmetric multilinear function $\Vol_{\Gamma}: \underbrace{\mathcal{M}_{\Gamma} \times ... \times \mathcal{M}_{\Gamma}}_{n} \rightarrow \mathbb{R}$, such that $\Vol_{\Gamma}((A,B),...,(A,B)) = \V(A,B)$ for every pair $(A,B) \in \mathcal{M}_{\Gamma}$.
\end{defn}

\begin{remark*}
Similar to the classical mixed volumes, the mixed volume of pairs of polyhedra satisfies the formula:
\begin{align*}
\Vol_\Gamma((A_1,B_1),\ldots, (A_n,B_n)) = \sum_{I \subseteq \{1,\ldots, n\}}(-1)^{n-|I|}\V(\sum_{i \in I} (A_i,B_i))
\end{align*}
where $(A_i,B_i) \in \mathcal{M}_{\Gamma}$ for $i$ from $1$ to $n$. In particular, by this formula, the mixed volume of pairs of lattice polyhedra is an integer (see \cite{rel6} Lemma 3 (4)).
\end{remark*}
Proof of existence of the mixed volume was given in \cite{rel6}, Lemma 3(1-3). In this research, we will primarily be focused on the case where the support cone is the positive orthant in $\mathbb{R}^n$, denoted by $C$. In particular, we shall be mostly interested in the pairs of polyhedra of the form $(C, B)$. 

\begin{convention*}
In the rest of the paper, all unbounded polyhedra are assumed to have bounded difference with the positive orthant $C$, unless otherwise stated.
\end{convention*}

\begin{defn}\label{def1}\cite{interl}
{
    Given $n$ polyhedra $B_1,...,B_n$ in $\R^n$, $B_1,...,B_n$ are called \textit{interlaced} if for every $k = 0, ... , n-1$, every $k$-dimensional face of $B = \mathrm{conv}(B_1 \cup ... \cup B_n)$ intersects at least $k+1$ polyhedra of $B_1, ... B_n$.
}
\end{defn}

\begin{exmp}
As shown in Figure 1, polyhedra $A$ and $B$ are interlaced in this case, because $\mathrm{conv}(A \cup B) = A$ and every $k$-face of $A$ intersects at least $k+1$ polyhedra. For example, the face $\mathrm{conv}(\{(2,0),(0,2)\})$ intersects both $A$ and $B$, and the face $(2,0)$ intersects $A$.
\end{exmp} 

There are some unique properties which drastically simplify the calculation of mixed volume of interlaced polyhedra:

\begin{theorem}\cite{interl}\label{int}
Given $n$ polyhedra $B_1,...,B_n$ in $\R^n$ and $B = \mathrm{conv}(B_1 \cup ... \cup B_n)$, if $B_1,...,B_n$ in $\R^n$ are interlaced, then $\Vol_C((C,B_1),...,(C,B_n)) = \Vol(C\backslash B)$. Otherwise, $\Vol_C((C,B_1),...,(C,B_n)) > \Vol(C\backslash B)$.
\end{theorem}

\begin{exmp}
Since $A$ and $B$ in Figure 1 are interlaced, we conclude that the mixed volume of the pairs $(C,A), (C,B)$ is simply $\Vol(C\backslash \mathrm{conv}(A \cup B)) = \Vol(C \backslash A) = 4$.
\end{exmp}

On the other hand, we will introduce another theorem to calculate the mixed volume in the general case. 

The lattice volume of polytopes in a rational subspace as in Definition \ref{lat_vol} and the mixed volume in Definition \ref{mixed_top} give rise to the mixed volume of polytopes in a rational subspace.

\begin{defn}
Given a $n-$dimensional rational vector subspace $V$, the \textit{mixed volume} of polytopes parallel to $V$ is the unique function of $n$ polytopes contained in affine spaces parallel to $V$, satisfying the following properties: it is symmetric, multilinear, invariant under shifts of arguments, and equal to the lattice volume once its arguments are equal.
\end{defn}

\begin{theorem}\label{mixv}\cite{rel6}
Given $n$ polyhedra $B_1,...,B_n$ in $\R^n$, we have 
$$\Vol_C((C,B_1),...,(C,B_n)) = \sum_\gamma B_1(\gamma) \Vol^{(n-1)}(B_2^\gamma, ... , B_n^\gamma)$$
where $B_1(\gamma)$ is the support function (Definition \ref{support}), and $\gamma$ runs over all primitive covectors with positive coordinates, and $\Vol^{(n-1)}$ is the mixed volume of $(n-1)$ bounded polytopes parallel to $(n-1)$-dimensional subspace $\ker (\gamma)$. The set of all $x\in B_i$ such that $\gamma(x)=B_i(\gamma)$ is a face of $B_i$, it will be called the support face of $\gamma$ and denoted by $B_i^\gamma$.
\end{theorem}

This theorem easily follows from its well known metric version (see e.g. Theorem 4.10 in \cite{thm2}).  For our relative version, see Section 4 in \cite{rel6}.

\begin{exmp}
In order to calculate $\Vol(A,B)$ in Figure 1, observe that we only need to focus on $\gamma$ such that $\Vol^{(n-1)}(B^\gamma) \neq 0$. There are only two such covectors, ie $\gamma_1 = (2,1)$ and $\gamma_2 = (1,2)$, which correspond to edge  $\mathrm{conv}(\{(0,3),(1,1)\})$ and $\mathrm{conv}(\{(1,1),(3,0)\})$ of $B$, respectively. As such, we have $\Vol^{(n-1)}(B^{\gamma_1}) = \Vol^{(n-1)}(B^{\gamma_2}) = 1$. On the other hand, by the definition of support function, we have $A(\gamma_1) = A(\gamma_2) = 2$.

Correspondingly, we have $\Vol_C((C,A),(C,B)) = 1 \cdot 2 + 1 \cdot 2 = 4$, which is identical with our previous result. 
\end{exmp}

\section{Tuples of polyhedra with mixed volume 1}\label{Smv1}

We denote the positive orthant in $\R^n$ by $C$. The hyperplane passing through the points $(a_{1},0,\ldots,0),(0,a_{2},0,\ldots,0),...,(0,\ldots,0,a_{n})$, where $a_i > 0$ for all $i \in \{1\ldots,n\}$, splits it into a simplex and an unbounded polyhedron that we denote by $C_{a_{1},a_{2},...,a_{n}}$. In the rest of the paper, $e_m$ denotes the point $(0,\ldots,0,1,0,\ldots,0)$ where $1$ is in the $m^{th}$ place. Recall that unless otherwise stated, all unbounded polyhedra in this paper are assumed to have bounded difference with the positive orthant $C$.

\begin{theorem}\label{v1} Consider $n$ $n$-dimensional lattice polyhedra $B_1, \ldots, B_n$. Then, up to changing the order of polyhedra and coordinates, the following two statements are equivalent:

1) the mixed volume of the pairs $(C,B_1),\ldots,(C,B_n)$ is well defined and equal to 1.

2) For every $m$ from $1$ to $n$, $e_m \in B_m$.
\end{theorem}

We shall need the following combinatorial lemma in finishing the proof. 

\begin{lemma}
If the mixed volume of lattice pairs $(C,B_1),...,(C,B_n)$ in $\R^n$ is equal to 1, then every point of $\{e_1, ... , e_n\}$ belongs to at least one of $B_1, ... , B_n$. Moreover, all of the polyhedra $B_1, ... ,B_n$ contain at least one of those points. 
\end{lemma}
\begin{proof}
First, without loss of generality, assume $e_1$ is not contained in any of the polyhedra in $B_1, ... , B_n$. Then, given $B = \mathrm{conv}(B_1, ... ,B_n)$, $e_1$ does not belong to $B$. Since the mixed volume of all pairs is $1$, none of $B_1,...,B_n$ can contain the origin, so $B$ must contain a vertex in the form of $(a,0,...,0)$ where $a \geq 2$. Thus, the lattice volume of $C\backslash B$ is at least $2$, and, by Theorem \ref{int}, the mixed volume of $(C,B_1),...,(C,B_n)$ is at least $2$, which contradicts the hypothesis. 

Now, without loss of generality, assume $B_1$ does not contain any of $\{e_1, ... , e_n\}$. Thus, $B_1 \subset C_{2,...,2}$. $\Vol_C((C,B_1),...,(C,B_n)) \geq \Vol_C((C,C_{2,...,2}),...,(C,B_n))$. By the linearity of mixed volume, $\Vol_C((C,C_{2,...,2}),...,(C,B_n)) = 2\Vol_C((C,C_{1,...,1}),...,(C,B_n)) \geq 2$. As such, we have $\Vol_C((C,B_1),...,(C,B_n)) \geq 2$, which contradicts the original claim.
\end{proof}

With the lemma and Frobenius-König Theorem, we are ready to prove Theorem \ref{v1}.

\begin{theorem} (Frobenius-König Theorem)\cite{frob}\label{Konig}
The permanent of an $n \times n$ integer matrix with all entries either $0$ or $1$ is $0$ if and only if the matrix contains an $r \times s$ submatrix of $0$s with $r+s=n+1$.
\end{theorem}

\begin{proof}[proof of Theorem \ref{v1}]
First, assume the mixed volume of the pairs $(C,B_1),\ldots,(C,B_n)$ equals 1. For the same reason stated in the proof of Lemma 1, none of $B_1, ... , B_n$ could contain the origin. Thus, given that $B = \mathrm{conv}(B_1 \cup ... \cup B_n)$, we have $\Vol(C\backslash B) \geq 1$. Since mixed volume of the pairs is 1, we have $\Vol(C\backslash B) = 1$, and B is $C \backslash W$, where $W$ is the convex hull of all points in $\{e_1, ... , e_n\}$ and the origin. Since $W$ is a simplex, convex hull of all subsets of $\{e_1, \ldots , e_n\}$ are faces of $W$, and thus, faces of $B$. By Theorem \ref{int}, $B_1, \ldots , B_n$ must be interlaced. For the sake of contradiction, assume there is no order of pairs of polytopes that could satisfy the relationship stated in Theorem \ref{v1}. Consider a matrix with $n \times n$ with $1$ on the $i,j-$th entry if $e_i \in B_j$ and $0$ otherwise. By Frobenius-König Theorem, there must exist a $r \times s$ submatrix of $0$ with $r+s = n+1$.  As such, $s$ of $B_1, \ldots , B_n$, say $B_1, \ldots , B_{s}$, do not contain $n-s+1$ elements of $\{e_1, \ldots , e_n\}$, say $e_1, \ldots , e_{n-s+1}$. Since $B_1, \ldots , B_s$ do not contain any point in $e_{1}, \ldots , e_{n-s+1}$, by convexity, they do not intersect the $n-s$ dimensional face $\{e_{1}, \ldots , e_{n-s+1}\}$. Correspondingly, the $n-s$ dimensional face $\{e_{1}, \ldots , e_{n-s+1}\}$ intersects at most $n-s$ polytopes in $B_1, \ldots , B_n$, namely, $B_{s+1}, \ldots , B_{n}$. Thus, by definition, $B_1, \ldots , B_n$ are not interlaced. Thus, by Theorem \ref{int}, mixed volume of $(C,B_1), \ldots ,(C,B_n)$ is greater than $1$. 

If $e_m \in B_m$ for every $m$ from $1$ to $n$, then $C \backslash B$ is the standard simplex with volume 1. Moreover, for every $k \in [1,n]$, every $k$-dimensional face of B intersects $k$ polyhedra from $B_1, ... , B_n$, one on each vertex. As such, $B_1, ... , B_n$ are interlaced, and $\Vol_C((C,B_1),...,(C,B_n)) = 1$.
\end{proof}

\section{Tuples of polyhedra with mixed volume 2}\label{Smv2}

\begin{defn}
A tuple of $n$-dimensional polyhedra $B_1, \ldots , B_n$ is said to be \textit{minimal} with mixed volume $V$, if $\Vol_C((C,B_1),...,(C,B_n)) = V$, and for every $i$ from $1$ to $n$ and point $\delta \in C\backslash B_i$, $\Vol_C((C,B_1),...,(C,\mathrm{conv}(B_i \cup \delta)),...,(C,B_n)) < V$.
\end{defn}

 Recall that unless otherwise stated, all unbounded polyhedra in this paper are assumed to have bounded difference with the positive orthant $C$.
 
\begin{theorem}\label{v2}
Consider $n$ $n$-dimensional lattice polyhedra $B_1, \ldots , B_n$. If the tuple $B_1, \ldots , B_n$ is minimal by inclusion with mixed volume 2, then up to changing the order of polyhedra and coordinates, $B_1, \ldots , B_n$ is in the form of $k$ copies of 
$C_{1,\ldots,1,2,\ldots,2}$ ($(k-1)$ many 1's) and $n-k$ copies of $C_{1,\ldots,1}$ for $k \in \{1, \ldots , n\}$.
\end{theorem}

We will first prove those tuples of polyhedra are minimal polyhedra with mixed volume 2:
\begin{lemma}
The tuples of polyhedra described in Theorem \ref{v2} are minimal polyhedra with mixed volume 2.
\end{lemma}

\begin{proof}
First, recognize that the polyhedra described above are minimal polyhedra. Each of those tuples has finitely many possibilities for the choice of an additional vertex $\delta$. For every choice of $\delta$, the resulting tuple has mixed volume $0$ or $1$ by Theorem \ref{v1}.

Second, we realize that the tuples of polyhedra described in Theorem \ref{v2} satisfy $\Vol_C((C,B_1),...,(C,B_n)) = V$. Consider the $m^{\text{th}}$ case, where $m \in [1,n]$, the polyhedra are shown below. 

$B_1, \ldots , B_m = C_{1, \ldots ,1,2, \ldots ,2}, B_{m+1}, \ldots , B_n = C_{1, \ldots ,1}$, where in $C_{1,\ldots,1,2,\ldots,2}$, the first $m-1$ terms are 1, and the following $n-m+1$ terms are $2$. If the coordinates of $B_{m+1}, ... , B_n$ are doubled to $C_{2, ... ,2}$, denoted as $B'_{m+1}, ... , B'_n$, claim $B_1, ... , B_m, B'_{m+1}, ... , B'_n$ are interlaced. 

Realize that $\mathrm{conv}(B_1, ... , B_m, B'_{m+1}, ... , B'_n) = C_{1,...,1,2,...,2}$, where the first $m - 1$ terms are $1$, and last $n-m+1$ terms are 2. For every $k$-dimensional bounded face $f$ with $k$ vertices from $\{e_1, ... , e_{m-1}, 2e_m, ... ,2e_n\}$, we will show that $f$ intersects at least $k+1$ polyhedra of $B_1, ... , B_m, B'_{m+1}, ... , B'_n$. For $v \in \{e_1, ... , e_{m-1}\}$, it intersects all $m$ polyhedra of $B_1, ... , B_m$. For $v' \in \{2e_m, ... , 2e_n\}$, it intersects all $n$ polyhedra. As such, if $\text{dim}(f) \leq m-1$, $f$ intersects at least $m$ polyhedra. If $\text{dim}(f) \geq m$, $f$ intersects all $n$ polyhedra. Correspondingly, $B_1, ... , B_m, B'_{m+1}, ... , B'_n$ are interlaced. By Theorem \ref{int}, the corresponding mixed volume is $\Vol(C \backslash \mathrm{conv}(B_1, ... , B_m, B'_{m+1}, ... , B'_n)) = 2^{n-m+1}$.

By linearity of mixed volume, 
\begin{align*}
    2^{n-m}&\Vol_C((C,B_1),\ldots,(C,B_n)) \\
    &= \Vol_C((C,B_1), \ldots ,(C,B_m),(C,2B_{m+1}), \ldots ,(C,2B_{n}))\\
    &= \Vol_C((C,B_1), \ldots ,(C,B_m),(C,B'_{m+1}), \ldots ,(C,B'_{n})) \\
    &= 2^{n-m+1}
\end{align*}
Correspondingly, $\Vol_C((C,B_1),\ldots,(C,B_n))=2$.
\end{proof}

\begin{proof} [proof of Theorem \ref{v2}]
For similar reason as in Lemma 1, none of $B_1,...,B_n$ can contain the origin. Assume that up to change of order of polyhedra and coordinates, there is a tuple of polyhedra $B'_1, ... , B'_n$ with $\Vol_C((C,B'_1),\ldots,(C,B'_n)) = 2$ that is not contained in the minimal tuple in described in Theorem \ref{v2} where $k = n$, ie, for every $e_m \in \{ e_1, ... ,e_n \}$, $e_m$ belongs to at least one of $B'_1, ... , B'_n$. On the other hand, since the mixed volume is greater than $1$, $B'_1, ... , B'_n$ can not satisfy the constraints offered in Theorem \ref{v1}. Thus, by Frobenius-König Theorem , upon changing the order of the polyhedra, there must exist a $m \in [1,n]$, such that $B'_1, ... ,B'_m$ cover less than $m$ elements in $\{ e_1, ... ,e_n \}$. Correspondingly, one of the following statements must be true.

1) $B'_1$ contains none of $\{ e_1, ... ,e_n \}$;

2) $B'_1 \cup B'_2$ contains at most 1 element of $\{ e_1, ... ,e_n \}$;

3) $B'_1 \cup B'_2 \cup B'_3$ contains at most 2 elements of $\{ e_1, ... ,e_n \}$;

...

n-1) $B'_1 \cup ... \cup B'_{n-1}$ contains at most $n-2$ elements of $\{ e_1, ... ,e_n \}$;

Clearly, up to change of order of polyhedra and coordinates, each tuple described above with index $i$ is contained in the minimal tuple described in the theorem with $k = i$. Thus, the $n$ tuples describe in Theorem \ref{v2} cover all possible combinations of minimal tuples of polyhedra with mixed volume 2. 
\end{proof}

\section{Finite number of polyhedra with finite mixed volume} \label{final}
\begin{theorem}\label{finv}
If $B_1,...,B_n$ are $n$-dimensional minimal polyhedra with mixed volume $V$, then $C_{V,...,V} \subseteq B_1,...,B_n$.
\end{theorem}

 Again, recall that we assume $ B_1,...,B_n$ above have bounded difference with $C$. Moreover, notice that this theorem is tight, ie, $C_{V,...,V,V-1}$ is not contained in all tuples of polyhedra with mixed volume $V$.
\begin{exmp}
$B_1,...,B_n$ with $B_1 = C_{V,...,V}$ and $B_2,...,B_n = C_{1,...,1}$ satisfy $\Vol_C((C,B_1),...,(C,B_n)) = V$. However, $C_{V,...,V,V-1} \not\subset B_1$.
\end{exmp}

In order to prove Theorem \ref{finv}, we will first prove the two following lemmas:

\begin{lemma}
If $\Vol_C((C,B_1),..., (C,B_n)) = V$, and $B'_1$ is the convex hull of $B_1$ and the point $c = V \cdot e_1$, then $\Vol_C((C,B'_1),(C,B_2),...,(C,B_n)) = V$ as well.
\end{lemma}

\begin{proof}
By monotonicity of the mixed volume, we have $\Vol_C((C,B'_1), (C,B_2),...,(C,B_n)) \leq V$.

It remains to prove that $\Vol_C((C,B'_1), (C,B_2),...,(C,B_n))  \geq V$. We shall do so by computing $\Vol_C((C,B'_1), (C,B_2),...,(C,B_n))$ and $\Vol_C((C,B_1), (C,B_2),...,(C,B_n))$ with Theorem \ref{mixv}, and comparing every term in both expressions.

The terms are parameterized by integer linear functions $\gamma$ on $\mathbb{R}^n$ with positive coefficients. Let us pick an arbitrary $\gamma$ and prove that $B'_1(\gamma)\Vol^{(n-1)}(B_2^\gamma ,\ldots, B_n^\gamma ) \geq B_1(\gamma)\Vol^{(n-1)}(B_2^\gamma ,\ldots, B_n^\gamma )$.

First of all, if $\Vol^{(n-1)}(B_2^\gamma , ..., B_n^\gamma) = 0$, the latter inequality is trivially satisfied. So we assume that $\Vol^{(n-1)}(B_2^\gamma , ..., B_n^\gamma ) > 0$ and aim at proving $B'_1(\gamma) \geq B_1(\gamma)$.

Since $\Vol_C((C,B_1),...,(C,B_n)) = V$, by the formula of Theorem \ref{mixv}, we know $B_1(\gamma) \leq V$ when $\Vol^{(n-1)}(B_2^\gamma,...,B_n^\gamma) \geq 1$. If the linear function $\gamma$ attains its minimum on the polyhedron $B'_1$ at point $c$, then $B'_1(\gamma) \geq V \geq B_1(\gamma)$. Otherwise, if $B'_1(\gamma)$ does not attain its minimum at the point $c$, $B'_1(\gamma) = B_1(\gamma)$. As a result, in both cases, $B'_1(\gamma) \geq B_1(\gamma)$, so by Theorem \ref{mixv}, $\Vol_C((C,B'_1), (C,B_2),...,(C,B_n)) \geq V$. 

Correspondingly, $\Vol_C((C,B'_1),(C,B_2),...,(C,B_n)) = V$.
\end{proof}

\begin{lemma}
If $B_1,...,B_n$ are $n$-dimensional minimal polyhedra with mixed volume $V$, and $B'_1 = \mathrm{conv}(C_{V,...,V} \cup B_1)$, then $\Vol_C((C,B'_1),...,(C,B_n)) = V$. 
\end{lemma}

\begin{proof}
Since $\mathrm{conv}(B_{V,...,V} \cup B_1) = \mathrm{conv}(\cup_{i=1}^n (V \cdot e_i)\cup B_1)$, applying Lemma 3 several times, we have $\Vol_C((C,B'_1), (C,B_2),...,(C,B_n)) = V$. 
\end{proof}

\begin{proof}[proof of Theorem \ref{finv}]
Assume there exists a tuple of minimal polyhedra $B_1, \ldots ,B_n$ with mixed volume $V$, such that up to change of order of polyhedra and coordinates, $B_1$ does not contain $C_{V, \ldots ,V}$. Denote $B'_1 = \mathrm{conv}(C_{V, \ldots ,V} \cup B_1)$. By definition of minimal polyhedra, $\Vol_C((C,B'_1),(C,B_2), \ldots ,(C,B_n)) < \Vol_C((C,B_1), \ldots ,(C,B_n)) = V$. \\
However, by Lemma 4, $\Vol_C((C,B'_1),(C,B_2), \ldots ,(C,B_n)) = V$, which is a contradiction. 
Thus, $C_{V, \ldots ,V} \subset B_1, \ldots ,B_n$ for all minimal polyhedra.
\end{proof}

Since there are only finitely many tuples of polyhedra such that $C_{V, \ldots ,V} \subseteq B_1, \ldots ,B_n$ for any given dimension $n$, we get the following conclusion:
\begin{theorem}
 There are finitely many minimal by inclusion tuples of polyhedra for a given mixed volume.
\end{theorem}

\end{document}